\documentclass[a4paper]{article}

\usepackage[top=3.3cm, bottom=3.3cm, left=3cm, right=3cm]{geometry}

\usepackage[utf8]{inputenc}
\usepackage{bm}
\usepackage{amsfonts}
\usepackage{amsmath,amssymb,amsthm}
\usepackage{comment}
\usepackage{enumitem}
\usepackage{cite}
\usepackage{tikz}
\usepackage{color}
\usepackage{ulem}

\newtheorem{theorem}{Theorem}
\newtheorem*{theorem*}{Main Theorem}
\newtheorem{lemma}[theorem]{Lemma}
\newtheorem*{settingRS*}{Setting RS}

\providecommand{\norm}[1]{\lVert#1\rVert}

\numberwithin{equation}{section}

\title{Curve shortening flows on rotational surfaces \\generated by monotone convex functions}
\author{Naotoshi Fujihara}
\date{}

\begin{document}

\maketitle

\begin{abstract}
    In this paper, we study curve shortening flows on rotational surfaces in $\mathbb{R}^3$. We assume that the surfaces have negative Gauss curvatures and that some condition related to the Gauss curvature and the curvature of embedded curve holds on them. Under these assumptions, we prove that the curve remains a graph over the parallels of the rotational surface along the flow. Also, we prove the comparison principle and the long-time existence of the flow.
\end{abstract}


\section{Introduction}
\label{sec:Introduction}

Let $f_t \colon \mathbb{S}^1 \to \overline{M}$, $t \in [0,T)$ be a smooth family of immersions of a circle $\mathbb{S}^1$ into a $2$-dimensional complete oriented Riemannian manifold $(\overline{M},\overline{g})$. 
"Smooth family" means that the map $F \colon \mathbb{S}^1 \times [0,T) \to \overline{M}$ defined by $F(x,t)=f_{t}(x)$ is smooth. 
Let $s$ be an arc-length parameter, then we get $\partial_s = \frac{1}{\norm{df_t(\partial_x)}} \, \partial_x$. 
Set $\mathfrak{t}=d f_t(\partial_s)$, which is a unit tangent vector of $f_t (\mathbb{S}^1)$, we take a unit normal vector $N_t$ of $f_t$ as $\{ \mathfrak{t},N_t \}$ gives the orientation of $\overline{M}$. $\{f_t\}_{t \in [0,T)}$ is called a \textit{curve shortening flow} (CSF) if it is a solution of the equation 
\begin{equation}
    \frac{\partial F}{\partial t} = - \kappa_t \,N_t, \label{eq:CSF}
\end{equation}
where $\kappa_t$ is the curvature of $f_t$ with respect to $-N_t$. 
CSF is a special case of mean curvature flow where the dimension of an immersed manifold is one.
It is known that the equation $(\ref{eq:CSF})$ has a unique solution in short-time. In this paper we study its behavior.

The CSF has been studied on various ambient spaces. In the Euclidean plane, M. Gage and R. Hamilton \cite{gage1986heat} showed that the any embedded convex closed curve shrinks to a round point. After that, M. Grayson \cite{grayson1987heat} proved that the embedded closed curves become convex. From these two results, the any embedded closed curve collapses to a round point. The behaviors of CSF are also studied in surfaces \cite{grayson1989shortening,gage1990curve}, warped product manifolds \cite{zhou2017curve} and compact Riemannian manifolds \cite{ma2007curve}.

In this paper, we consider the case that the ambient space $\overline{M}$ is a rotational surface in $\mathbb{R}^3$ defined by 
\begin{equation*}
    \overline{M}=\{ (r(z) \cos{\theta},r(z) \sin{\theta},z) \,\lvert\, \theta \in [0,2 \pi) , z \in I \,\},
\end{equation*}
where $I=(a,b)$ is an open interval and $r \colon I \to \mathbb{R}$ is a smooth positive function. The metric $\overline{g}$ of $\overline{M}$ is defined by $\overline{g}= \iota^{*}g_{\mathbb{E}^3}$, that is, the metric $\overline{g}$ is induced by the inclusion $\iota \colon \overline{M} \to \mathbb{E}^3$. 
The Gauss curvature is denoted by $\overline{K}$ and the curvatures of $\theta$-curves (given by (\ref{eq:standard curvature})) is denoted by $\mathcal{K}$ and then we consider the following setting:
\begin{settingRS*}
Let $I=(-\infty,b)$ and $J=(- \infty, \widetilde{b}) \subseteq I$. Then for every $z \in J$, $(\overline{M},\overline{g})$ satisfies 
\begin{align*}
0 < \dot{r}(z) \leq \frac{1}{\sqrt{2}}, \quad  \left( 2 \, \mathcal{K}^2 + \overline{K} \right)(z) < 0.
\end{align*}
\end{settingRS*}
The second condition is equivalent to 
\begin{align}
\label{setting:RS}
r(z) \ddot{r}(z) > 2 \dot{r}(z)^2 (\dot{r}(z)^2 +1).     
\end{align}
The examples of $r(z)$ and $J=(-\infty,\widetilde{b})$ that meet Setting RS are below:

\begin{itemize}
    \item When $r(z)=(-z)^{-\alpha}$, $z<0$, $0<\alpha<1$, and $\widetilde{b} = \min\{-(2 \alpha^2)^{\frac{1}{2(\alpha+1)}}, -(\frac{2 \alpha^3}{1 - \alpha})^{\frac{1}{2(\alpha + 1)}}\}$, $(\overline{M},\overline{g})$ satisfies the setting.
    \item When $r(z)=e^{-\frac{1}{z}}$ and $\widetilde{b}$ is a sufficiently small negative number, $(\overline{M},\overline{g})$ satisfies the setting.  
\end{itemize}
If $r(z)=(-z)^{-\alpha}$, $z<0$, $\alpha \geq 1$, then the inequality (\ref{setting:RS}) does not hold (see Figure \ref{fig:r(z)}).


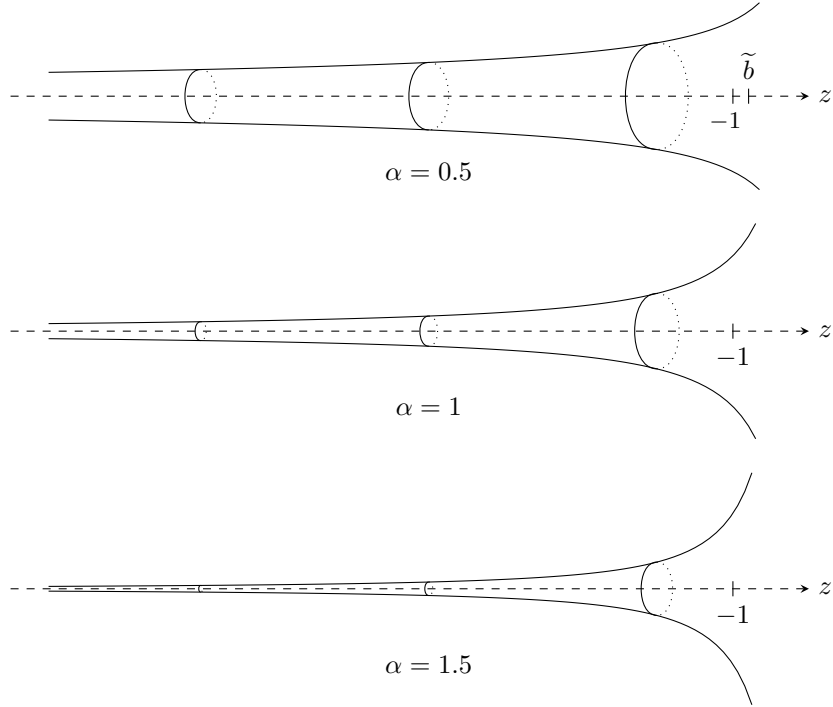
\begin{figure}[htbp]
    \centering

\begin{tikzpicture}[domain=-10:-0.65,samples=100,>=stealth]
\draw[->,dashed] (-10.5,0) -- (0,0) node[right] {$z$};

\draw (-0.7937,-0.1) -- (-0.7937,0.1) node[above] {$\widetilde{b}$};

\draw (-1,0.1) -- (-1,-0.1);
\draw (-1.1,-0.1)  node[below] {\small{\textcolor{black}{$-1$}}};

\draw plot (\x, {pow(-\x,-0.5)});
\draw plot (\x, {-pow(-\x,-0.5)});

\draw (-2,{pow(2,-0.5)}) to [out=180,in=180] (-2,{-pow(2,-0.5)});
\draw[dotted] (-2,{pow(2,-0.5)}) to [out=0,in=0] (-2,{-pow(2,-0.5)});

\draw (-5,{pow(5,-0.5)}) to [out=180,in=180] (-5,{-pow(5,-0.5)});
\draw[dotted] (-5,{pow(5,-0.5)}) to [out=0,in=0] (-5,{-pow(5,-0.5)});

\draw (-5,-1) node[] {$\alpha=0.5$};

\draw (-8,{pow(8,-0.5)}) to [out=180,in=180] (-8,{-pow(8,-0.5)});
\draw[dotted] (-8,{pow(8,-0.5)}) to [out=0,in=0] (-8,{-pow(8,-0.5)});

\end{tikzpicture}

\vspace{4mm}

\begin{tikzpicture}[domain=-10:-0.7,samples=100,>=stealth]
\draw[->,dashed] (-10.5,0) -- (0,0) node[right] {$z$};

\draw (-1,0.1) -- (-1,-0.1) node[below] {$-1$};

\draw plot (\x, {pow(-\x,-1)});
\draw plot (\x, {-pow(-\x,-1)});

\draw (-2,{pow(2,-1)}) to [out=180,in=180] (-2,{-pow(2,-1)});
\draw[dotted] (-2,{pow(2,-1)}) to [out=0,in=0] (-2,{-pow(2,-1)});

\draw (-5,{pow(5,-1)}) to [out=180,in=180] (-5,{-pow(5,-1)});
\draw[dotted] (-5,{pow(5,-1)}) to [out=0,in=0] (-5,{-pow(5,-1)});

\draw (-5,-1) node[] {$\alpha=1$};

\draw (-8,{pow(8,-1)}) to [out=180,in=180] (-8,{-pow(8,-1)});
\draw[dotted] (-8,{pow(8,-1)}) to [out=0,in=0] (-8,{-pow(8,-1)});

\end{tikzpicture}  

\vspace{4mm}

\begin{tikzpicture}[domain=-10:-0.75,samples=100,>=stealth]
\draw[->,dashed] (-10.5,0) -- (0,0) node[right] {$z$};

\draw (-1,0.1) -- (-1,-0.1) node[below] {$-1$};

\draw plot (\x, {pow(-\x,-1.5)});
\draw plot (\x, {-pow(-\x,-1.5)});

\draw (-2,{pow(2,-1.5)}) to [out=180,in=180] (-2,{-pow(2,-1.5)});
\draw[dotted] (-2,{pow(2,-1.5)}) to [out=0,in=0] (-2,{-pow(2,-1.5)});

\draw (-5,{pow(5,-1.5)}) to [out=180,in=180] (-5,{-pow(5,-1.5)});
\draw[dotted] (-5,{pow(5,-1.5)}) to [out=0,in=0] (-5,{-pow(5,-1.5)});

\draw (-5,-1) node[] {$\alpha=1.5$};

\draw (-8,{pow(8,-1.5)}) to [out=180,in=180] (-8,{-pow(8,-1.5)});
\draw[dotted] (-8,{pow(8,-1.5)}) to [out=0,in=0] (-8,{-pow(8,-1.5)});

\end{tikzpicture}
    
    \caption{Rotational surfaces generated by $r(z)=(-z)^{-\alpha}$}
    \label{fig:r(z)}
\end{figure}


In this paper, we will prove the following theorem.
\begin{theorem*}
Let $(\overline{M},\overline{g})$ be a rotational surface that satisfies Setting RS and $\{ f_t \}_{t\in[0,T)}$ be a CSF that satisfies the initial condition
\begin{align}
    \label{initial}
    f_0 \text{ is an embedding and satisfies } z_0(\mathbb{S}^1) \subseteq J \text{ and } u_0 > 0,
\end{align}
where $J$ is the open interval defined in Setting RS and $z_0$, $u_0$ are defined by $(\ref{eq:functions})$. $u_0 > 0$ means that the initial curve is a graph over the parallels of the rotational surface $($see Figure \ref{fig:u_0}$)$. Under these two conditions, we have the followings:
\begin{enumerate}
    \item[$(1)$] $T=\infty$.
    \item[$(2)$] $z_t \to -\infty$ $(t \to \infty)$.
    \item[$(3)$] $\kappa_t \to 0$ $(t \to \infty)$.
\end{enumerate}
\end{theorem*}

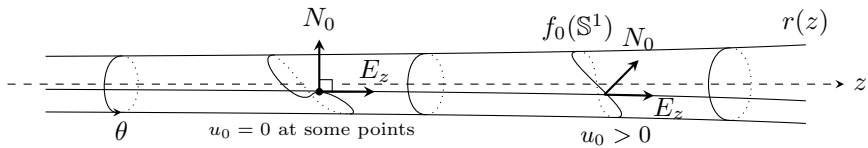
\begin{figure}[htbp]
    \centering

\begin{tikzpicture}[domain=-15:-5,samples=200,>=stealth]
\draw[->,dashed] (-11.1,0) -- (-4.5,0) node[right] {$z$};
\draw[dashed] (-15.5,0) -- (-11.5,0);

\draw plot (\x, {-0.7 + exp(-1/\x)}) node[above] {$r(z)$};
\draw plot (\x, {0.7 - exp(-1/\x)});

\draw plot (\x, {1.0 - exp(-1/\x)});

\draw (-6,{-0.7 + exp(1/6)}) to [out=180,in=180] (-6,{0.7 - exp(1/6)});
\draw (-10,{-0.7 + exp(1/10)}) to [out=180,in=180] (-10,{0.7 - exp(1/10)});
\draw[->] (-14,{-0.7 + exp(1/14)}) to [out=180,in=180] (-14,{0.7 - exp(1/14)}) node[below] {$\theta$};

\draw[dotted] (-6,{-0.7 + exp(1/6)}) to [out=0,in=0] (-6,{0.7 - exp(1/6)});
\draw[dotted] (-10,{-0.7 + exp(1/10)}) to [out=0,in=0] (-10,{0.7 - exp(1/10)});
\draw[dotted] (-14,{-0.7 + exp(1/14)}) to [out=0,in=0] (-14,{0.7 - exp(1/14)});

\draw (-8,{-0.7 + exp(1/8)}) node[above] {$f_0(\mathbb{S}^1)$};
\draw[] (-8,{-0.7 + exp(1/8)}) to [out=180,in=0] (-7.5,{0.7 - exp(1/7.5)});
\draw[dotted] (-8,{-0.7 + exp(1/8)}) to [out=0,in=180] (-7.5,{0.7 - exp(1/7.5)});
\draw  (-7.5,{0.7 - exp(1/7.5)}) node[below] {\small{$u_0 > 0$}};

\draw (-12,{-0.7 + exp(1/12)});
\draw[] (-12,{-0.7 + exp(1/12)}) to [out=180,in=225] (-11.5,{0.95 - exp(1/11.5)});
\draw[] (-11.5,{0.95 - exp(1/11.5)}) to [out=45,in=0] (-11,{0.7 - exp(1/11)});
\draw[dotted] (-12,{-0.7 + exp(1/12)}) to [out=325,in=180] (-11,{0.7 - exp(1/11)});
\draw (-11.5,{0.7 - exp(1/11.5)}) node[below] {\scriptsize{$u_0 = 0$ at some points}};

\fill[black] (-11.4,{0.99 - exp(1/11.5)}) circle (0.05);

\draw[->,thick] (-7.65,{1.0 - exp(1/7.65)}) -- (-7.2,{-0.83 + exp(1/7.2)}) node[above] {$N_0$};
\draw[->,thick] (-7.65,{1.0 - exp(1/7.65)}) -- (-7.0,{1.0 - exp(1/7.0)});
\draw (-6.8,{1.1 - exp(1/7.0)}) node[below] {$E_z$};

\draw[->,thick] (-11.4,{0.99 - exp(1/11.5)}) -- (-11.4,{-0.5 + exp(1/11.5)}) node[above] {$N_0$};
\draw[->,thick] (-11.4,{0.99 - exp(1/11.5)}) -- (-10.67,{0.99 - exp(1/11.5)}) node[above] {$E_z$};
\draw[] (-11.4,{1.15 - exp(1/11.5)}) -- (-11.23,{1.15 - exp(1/11.5)}) -- (-11.23,{0.99 - exp(1/11.5)});
\end{tikzpicture}
    
    \caption{The meaning of $u_0 > 0$}
    \label{fig:u_0}
\end{figure}

This paper is organized as follows. 
We begin with the basic formulae about rotational surfaces in Sect. \ref{sec:basic formulae} and then we compute some evolution equations in Sect. \ref{sec:CSF}. 
Sect. \ref{sec:preserve} contains the proof that the flow remains a graph. Comparison principle and long-time existence are proved in Sect. \ref{sec:comparison principle} and Sect. \ref{sec:long-time existence}.

\section{Basic formulae about rotational surfaces}
\label{sec:basic formulae}

In this section we give the basic geometric quantities of rotational surfaces. 
The covariant derivatives and Gauss curvature are used to compute the evolution equations in Sect. \ref{sec:CSF}. 

Let $\phi(r(z) \cos{\theta}, r(z) \sin{\theta}, z) = (\theta,z)$ be a local coordinate system, then the metric $\overline{g}$ is expressed as $\overline{g}=r(z)^2 \, d\theta^2 + (\dot{r}(z)^2+1) \, d z^2$. 
To compute the covariant derivatives, we will use the local orthonormal frame
\begin{align*}
    E_{\theta}=\frac{1}{r(z)} \, \partial_{\theta}, \quad E_z= \frac{1}{\sqrt{\dot{r}(z)^2+1}} \, \partial_z,
\end{align*}
and $\omega^{\theta}$, $\omega^z$ denote those duals, then we have
$\omega^{\theta} = r \, d \theta$, $\omega^z = \sqrt{\dot{r}^2 + 1} \, d z$.

\begin{lemma}
\label{lem:nabla bar}
For the Levi-Civita connection $\overline{\nabla}$ of $(\overline{M},\overline{g})$, we have
\begin{align*}
    &\overline{\nabla}_{E_{\theta}} E_{\theta} = - \frac{\dot{r}}{r\sqrt{\dot{r}^2 + 1}} \, E_z, \quad \overline{\nabla}_{E_z} E_{\theta} = 0, \\
    &\overline{\nabla}_{E_{\theta}} E_z = \frac{\dot{r}}{r\sqrt{\dot{r}^2 + 1}} \, E_{\theta}, \quad \overline{\nabla}_{E_z} E_z = 0. \\ 
\end{align*}
\end{lemma}

\begin{proof}
We compute the Cartan connection forms $\omega^{\theta}{}_{z}$, $\omega^{z}{}_{\theta}$ that satisfy
\begin{align*}
    d \omega^{\theta} = - \omega^{\theta}{}_{z} \wedge \omega^z, \quad
    d \omega^{z} = - \omega^{z}{}_{\theta} \wedge \omega^{\theta}, \quad \omega^{\theta}{}_{z}=-\omega^{z}{}_{\theta}.
\end{align*}
The exterior derivatives of $\omega^{\theta}$, $\omega^z$ are
\begin{align*}
    d \omega^{\theta} = - \frac{\dot{r}}{r \sqrt{\dot{r}^2 + 1}} \, \omega^{\theta} \wedge \omega^z, \quad d \omega^z = 0,
\end{align*}
thus we obtain $\omega^{\theta}{}_{z} = \frac{\dot{r}}{r \sqrt{\dot{r}^2 + 1}} \, \omega^{\theta}$. By using $\overline{\nabla}_{X} E_{\theta} = \omega^{z}{}_{\theta}(X) \, E_z$, $\overline{\nabla}_{X} E_{z} = \omega^{\theta}{}_{z}(X) \, E_{\theta}$, we can get the formulae. 
\end{proof}

\begin{lemma}
For Levi-Civita connection $\overline{\nabla}$ of $(\overline{M},\overline{g})$, the principal curvatures with respect to the outward unit normal vector are
\begin{align*}
    -\frac{1}{r\sqrt{\dot{r}^2 + 1}}, \quad \frac{\ddot{r}}{(\dot{r}^2 + 1)^{\frac{3}{2}}},
\end{align*}
and the Gauss curvature $\overline{K}$ satisfies
\begin{align*}
    \overline{K} = \overline{\mathrm{Ric}}(\overline{v},\overline{v})=-\frac{\ddot{r}}{r(\dot{r}^2 + 1)^2}
\end{align*}
where $\overline{v}$ is the unit tangent vector of $(\overline{M},\overline{g})$. $\overline{K}$ does not depend on $\theta$, so we sometimes write $\overline{K}(z)$. 
\end{lemma}

\section{Curve shortening flow and the evolution equations}
\label{sec:CSF}

We shall calculate the evolution equations in this section and give an example of CSF. 

Let $f_t \colon \mathbb{S}^1 \to \overline{M}$, $t \in [0,T)$ be a smooth family of immersions that satisfies the equation (\ref{eq:CSF}).
The Levi-Civita connection induced on $\mathbb{S}^1$ by $f_t$ is denoted by $\nabla$, 
and $\overline{\nabla}^F$ denotes the connection of $\overline{\nabla}$ induced by $F$.
Let us write $\overline{g} = \langle,\rangle$, then we have 
\begin{align*}
&N=\langle N,E_{\theta} \rangle \, E_{\theta} + \langle N, E_z \rangle \, E_z, \\
&\mathfrak{t}=\langle N,E_z \rangle \, E_{\theta} - \langle N,E_{\theta} \rangle \, E_z.
\end{align*}
The followings are derived from the definitions.

\begin{lemma}
\begin{align}
    &\nabla_{\partial_s} \partial_s = 0, \label{eq:nabla0} \\ 
    &\overline{\nabla}^{F}_{\partial_s} N = \kappa \, \mathfrak{t}, \quad \overline{\nabla}^{F}_{\partial_s} \mathfrak{t} = - \kappa \, N. \notag
\end{align}
\end{lemma}
By (\ref{eq:nabla0}), the laplacian $\Delta$ on $(\mathbb{S}^1,f_{t}^{*}\,\overline{g})$ satisfies $\Delta=\partial_s \partial_s$.
We then define two functions on $\overline{M}$.
\begin{align}
\label{eq:functions}
    z_t = \pi_z \circ f_t, \quad u_t = \langle N_t, E_z \rangle,
\end{align}
where $\pi_z \colon \overline{M} \to \mathbb{R}$ denotes the projection to $z$ component.
The curvatures of $\theta$-curves $\mathcal{K}(z)$ are given by
\begin{equation}
    \mathcal{K}(z) = \frac{\dot{r}(z)}{r(z)\sqrt{\dot{r}(z)^2+1}}. \label{eq:standard curvature}
\end{equation}

\begin{lemma}
\label{lem:laplacian}

With the above notation, we have the following formulae
\begin{align*}
    &\partial_s u = (-\kappa  + \mathcal{K} u ) \, \langle N, E_{\theta} \rangle , \quad
    \partial_s \langle N, E_{\theta} \rangle = -u \, \left( -\kappa \,  + \mathcal{K} \, u \right),\\
    &\partial_s z = d\pi_z(\mathfrak{t}) = -\frac{1}{\sqrt{\dot{r}^2 +1}} \, \langle N,E_{\theta} \rangle, \quad  \dot{\mathcal{K}}(z) = -\sqrt{\dot{r}^2 +1} \left( \mathcal{K}^2(z) + \overline{K}(z) \right),\\
    &\Delta u = -\partial_s \kappa \, \langle N,E_{\theta} \rangle -\mathcal{K} \, \kappa \, (1-u^2) +\left( 2 \mathcal{K}^2 + \overline{K} \right)(u-u^3) -u \, (\mathcal{K} u - \kappa)^2,\\
    &\Delta z = -\frac{\dot{r} \, \ddot{r}}{(\dot{r}^2+1)^2} \, (1-u^2) - \frac{1}{\sqrt{\dot{r}^2 + 1}} \, (\kappa \, u - \mathcal{K} \, u^2). \\
\end{align*}
\end{lemma}

\begin{proof}
First we have 
\begin{align*}
    \partial_s u 
    &= \langle  \overline{\nabla}^{F}_{\partial_s} N , E_z \rangle + \langle N, \overline{\nabla}^{F}_{\partial_s} E_z \rangle \\
    &= \kappa \, \langle \mathfrak{t}, E_z \rangle + \langle N, \overline{\nabla}_{\mathfrak{t}} E_z \rangle \\
    &= -\kappa \, \langle N, E_{\theta} \rangle + u \, \langle N, \overline{\nabla}_{E_{\theta}}E_z \rangle \\
    &= \left( -\kappa + \mathcal{K} \, u \right) \langle N, E_{\theta} \rangle.
\end{align*}
Then, from the equation $u^2 + \langle N, E_{\theta} \rangle^2 = 1$, we can calculate
\begin{align*}
    2 \langle N, E_{\theta} \rangle \, \partial_s \langle N, E_{\theta} \rangle = - 2 u \, \partial_s u = -2 u  \left( -\kappa + \mathcal{K} \, u \right) \langle N, E_{\theta} \rangle,
\end{align*}
then $\partial_s \langle N, E_{\theta} \rangle$ is obtained.
$\partial_s z$ and $\dot{\mathcal{K}}$ are followed by straightforward computations. We combine them and get $\partial_s \mathcal{K} = (\mathcal{K}^2 + \overline{K}) \langle N, E_{\theta} \rangle$.
Using the above equations, we get
\begin{align*}
    \Delta u 
    &= \partial_s \partial_s u \\
    &= \partial_s \left\{ \left( -\kappa + \mathcal{K} \, u \right) \langle N, E_{\theta} \rangle \right\} \\
    &= -\partial_s \kappa \, \langle N, E_{\theta} \rangle + \left( \mathcal{K}^2 + \overline{K} \right) u \left( 1 - u^2 \right) +\mathcal{K}\left( -\kappa + \mathcal{K} \, u \right) \left( 1 - u^2 \right)  - u \, \left( - \kappa + \mathcal{K} \, u \right)^2 \\
    &= -\partial_s \kappa \, \langle N,E_{\theta} \rangle -\mathcal{K} \, \kappa \, (1-u^2) +\left( 2 \mathcal{K}^2 + \overline{K} \right)(u-u^3) -u \, (\mathcal{K} u - \kappa)^2,
\end{align*}
and 
\begin{align*}
    \Delta z 
    &= \partial_s \partial_s z \\
    &= \partial_s \left\{ -\frac{1}{\sqrt{\dot{r}^2 +1}} \, \langle N,E_{\theta} \rangle  \right\} \\
    &= \frac{\dot{r} \ddot{r}}{(\dot{r}^2 +1)^{\frac{3}{2}}} \, \partial_s z \, \langle N, E_{\theta} \rangle - \frac{1}{\sqrt{\dot{r}^2 +1}} \left\{ -u \, \left( -\kappa \,  + \mathcal{K} \, u \right) \right\} \\
    &= -\frac{\dot{r} \, \ddot{r}}{(\dot{r}^2+1)^2} \, (1-u^2) - \frac{1}{\sqrt{\dot{r}^2 + 1}} \, (\kappa \, u - \mathcal{K} \, u^2).
\end{align*}
\end{proof}

Next we compute the evolution equations.

\begin{lemma}
\label{lem:evolution eq}
We have the evolution equations
\begin{align*}
    &\left(\partial_t - \Delta\right)z = \frac{\dot{r}\ddot{r}}{(\dot{r}^2 + 1)^2} \, (1-u^2) - \frac{\dot{r}}{r(\dot{r}^2 +1)} \, u^2, \\
    &\left(\partial_t - \Delta\right)u= -\left( 2 \mathcal{K}^2 + \overline{K} \right)(u - u^3) +u \, (\mathcal{K} u - \kappa)^2 ,\\
    &\left(\partial_t - \Delta\right)\kappa =  \kappa \, (\kappa^2 + \overline{K}), \quad
    \left(\partial_t - \Delta\right)\kappa^2 = -2(\partial_s \kappa)^2 + 2\kappa^2 \, (\kappa^2 + \overline{K}).
\end{align*}
\end{lemma}

\begin{proof}
$\kappa$ is the same as the mean curvature of $f_t$, so the formulae follow the ones of mean curvature flow.
The others, it is enough to calculate the derivatives with respect to time:
\begin{align}
    &\partial_t z = d \pi_{z} (- \kappa \, N) = - \frac{\kappa \, u}{\sqrt{\dot{r}^2 + 1}}, \label{eq:timediff_z} \\
    &\partial_t u = - \partial_s \kappa \, \langle N, E_{\theta} \rangle - \mathcal{K} \, \kappa \, (1-u^2). \notag
\end{align}
Combined with the Lemma $\ref{lem:laplacian}$, the equations are obtained.
\end{proof}


At the end of this section, we prepare a lemma before we give one example of CSF.

\begin{lemma}
\label{lem:graph curvature}
Let $f \colon \mathbb{S}^1 \to \overline{M}$ be immersed as a graph in some neighborhood of $x_{0} \in \mathbb{S}^1$. This means that $\phi \circ f \circ \psi^{-1} (x) = (x,z(x))$ holds by using some local coordinate system $\psi$ about $x_0$. 
Under this assumption, the curvature $\kappa$ of $f$ is written as
\begin{align*}
    \kappa = \frac{1}{(r^2 + (\dot{r}^2 +1) \dot{z}^2)^{\frac{3}{2}}} 
    \left\{ \frac{\dot{r} \, \dot{z}^2}{\dot{r}^2 +1} \left( 2 \, (\dot{r}^2 + 1) - r \, \ddot{r} \right) + \frac{r}{\sqrt{\dot{r}^2 + 1}} \left( r \, \dot{r} - \ddot{z} \, (\dot{r}^2 + 1) \right) \right\}.
\end{align*}
\end{lemma}

\begin{proof}
Let $w = \norm{d f (\partial_x)}=\sqrt{r^2 + (\dot{r}^2 + 1) \, \dot{z}^2}$ , we can write 
\begin{align*}
    &\mathfrak{t} = \frac{1}{w} \left( r \, E_{\theta} + \dot{z} \, \sqrt{\dot{r}^2 +1} \, E_z \right), \\
    & N = \frac{1}{w} \left( - \dot{z} \sqrt{\dot{r}^2 +1} \, E_{\theta} + r \, E_z \right).
\end{align*}
Moreover, by using the equations below, we can compute the curvature $\kappa$:
\begin{align*}
    \kappa = \langle \overline{\nabla}^{f}_{\partial_s} N, \mathfrak{t} \rangle, \quad \partial_s r=\frac{\dot{r} \, \dot{z}}{w}, \quad 
    \partial_s \left( \dot{z} \, \sqrt{\dot{r}^2 +1} \right) = \frac{1}{w} \left( \ddot{z} \, \sqrt{\dot{r}^2 + 1} + \frac{\dot{r} \, \ddot{r} \, \dot{z}^2}{\sqrt{\dot{r}^2 + 1}} \right).
\end{align*}
\end{proof}

Now let us define $F \colon \mathbb{S}^1 \times [0,T) \to \overline{M}$ as $\phi \circ F(x,t) = (x,z(t))$, and we can write the curvature $\kappa$ as  
\begin{align*}
    \kappa(x,t)= \frac{\dot{r}(z(t))}{r(z(t))\sqrt{\dot{r}(z(t))^2 + 1}} = \mathcal{K}(z(t))
\end{align*}
by using $\dot{z}=\ddot{z}=0$ and Lemma \ref{lem:graph curvature}.
Thus the equation (\ref{eq:CSF}) and 
\begin{align}
    \frac{d z}{d t}(t) = - \frac{\dot{r}(z(t))}{r(z(t))(\dot{r}(z(t))^2 + 1)} \label{eq:standard flow}
\end{align}
are equivalent. Hence set the initial condition $z_{0}(x)$ as a constant function, then we will get the solution that $z_{t}(x)$ is independent of $x \in \mathbb{S}^1$, that is, the flow remains the shape of circle. This flow is used in Sect. \ref{sec:long-time existence}.

\section{Preserving the graph property}
\label{sec:preserve}

We devote this section to show that the flow remains a graph. 
Before proving the theorem, we prepare Hamilton's Trick for it. This will play an important role in the proofs of the theorems.

\begin{lemma}[\cite{hamilton1986four,mantegazza2011lecture}]
\label{lem:Hamilton}
Let $M$ be a closed manifold and  
$\rho \colon M \times (0,T) \to \mathbb{R}$ be a $C^1$ function. We define $\rho_{\max}$, $\rho_{\min}$ by $\rho_{\max}(t)=\max_{x \in M} \rho(x,t)$, $\rho_{\min}(t)=\min_{x \in M} \rho(x,t)$. Then $\rho_{\max}$, $\rho_{\min}$ are locally Lipschitz functions and for every differentiable time $t \in (0,T)$, 
\begin{align*}
    \frac{d \rho_{\max}(t)}{d t} = \frac{\partial \rho(x,t)}{\partial t}, \quad 
    \frac{d \rho_{\min}(t)}{d t} = \frac{\partial \rho(x,t)}{\partial t}
\end{align*}
hold at $x \in M$ such that $\rho(\cdot,t)$ attains its maximum or minimum.
\end{lemma}

We recall that the curve is a graph over the parallels of the rotational surface if and only if $u>0$, that is,  $1 \leq v = \frac{1}{u} < \infty$. 
Therefore, we will show that $v$ is uniformly bounded. 
To prove this, we first need the evolution equation for $v$.

\begin{lemma}
Let $v=\frac{1}{u}$. We have
\begin{equation*}
    \left(\partial_t - \Delta\right)v= -\frac{2}{v} \, (\partial_s v)^2  + \left( 2 \mathcal{K}^2 + \overline{K} \right)(v- \frac{1}{v}) -\frac{1}{v}  \, (\mathcal{K}-\kappa v)^2.
\end{equation*}
\end{lemma}
\begin{proof}
Using the relations
\begin{align*}
    \partial_t v = - \frac{1}{u^2} \, \partial_t u, \quad \partial_s v = - \frac{1}{u^2} \, \partial_s u, \quad \Delta v = -\frac{1}{u^2} \, \Delta u + \frac{2}{u^3} \left( \partial_s u \right)^2,
\end{align*}
we can get the equation from Lemma $\ref{lem:evolution eq}$.
\end{proof}

We then prove the theorem.

\begin{theorem}
Assume that Setting RS and the initial condition $(\ref{initial})$ hold.
Then for some positive constant $c$, we have $v_t(x) \leq c$, that is, $f_t$ is a graph over the parallels of the rotational surface $\overline{M}$ for $t \in [0,T)$.
\end{theorem}

\begin{proof}
Let $z_{\max}(t)=\max_{x \in \mathbb{S}^1} z(x,t)$. 
When $(x,t)$ satisfies $z_{\max}(t)=z(x,t)$, we have $u(x,t)=1$. 
From the evolution equation for $z$ (Lemma \ref{lem:evolution eq}), Lemma \ref{lem:Hamilton} and $\dot{r}(z)>0$, we have the following inequality:
\begin{align*}
    \frac{d z_{\max}}{d t}(t) \leq  - \frac{\dot{r}}{r \, (\dot{r}^2 +1)}(z(x,t)) < 0.
\end{align*}
Therefore we obtain $z_{\max}(t) \leq z_{\max}(0)$ and namely $z_t(\mathbb{S}^1) \subseteq J$.

Then let $t_0 < T$ and $(x,t)$ be a point that $v$ attains its maximum on $\mathbb{S}^1 \times [0,t_0]$. 
We can assume $t > 0$ and from the evolution equation for $v$ we get
\begin{align*}
    0 &\leq \left(\partial_t - \Delta\right)v(x,t) \\
    &\leq \left( 2 \, \mathcal{K}^2 + \overline{K} \right)(v - \frac{1}{v})(x,t).
\end{align*}
Here we combine Setting RS with $v \geq 1$, then  $v(x,t)=1$ follows.
Thus we take the case of $t=0$ into consideration, 
\begin{align*}
    v(x,t) \leq \max_{x \in \mathbb{S}^1}v_0(x)
\end{align*}
holds. 
Since we choose any $t_0$, the theorem follows.
\end{proof}
We remark that from this theorem the flow $\{ f_t \}_{t \in [0,T)}$ remains an embedding. This fact is used implicitly in the later sections.

\section{Comparison principle}
\label{sec:comparison principle}
The comparison principle for mean curvature flow in Euclidean space is well-known.
In this section we are going to show the counterpart in the sense of the CSF on rotational surfaces under Setting RS and the initial condition $(\ref{initial})$.

\begin{theorem}
\label{thm:comparison}
Let $\overline{M}$ be a rotational surface that satisfies Setting RS and $\{f_t\}_{t\in[0,T)}$, $\{\widetilde{f}_t\}_{t\in[0,T)}$ be CSFs on $\overline{M}$ that satisfy the initial condition $(\ref{initial})$, and we set $z_{\max}(t) = \max_{x \in \mathbb{S}^1} z(x,t)$,  $\widetilde{z}_{\min}(t)=\min_{x \in \mathbb{S}^1}\widetilde{z}(x,t)$. 
If we have $z_{\max}(0) \leq \widetilde{z}_{\min}(0)$, then $z_{\max}(t) \leq \widetilde{z}_{\min}(t)$ follows.
\end{theorem}
\begin{proof}
We set $\rho(t)=\widetilde{z}_{\min}(t) - z_{\max}(t)$ and obtain the equation below by $(\ref{eq:timediff_z})$: 
\begin{align*}
    \frac{d \rho}{d t}= -\frac{\widetilde{\kappa}}{\sqrt{\dot{r}(\widetilde{z}_{\min})^2 +1}} + \frac{\kappa}{\sqrt{\dot{r}(z_{\max})^2 +1}}.
\end{align*}
We choose any $t_1 > 0$ and let $x_1$, $\widetilde{x}_1$ be points that satisfy
$z_{\max}(t_1) = z(x_1,t_1)$, $\widetilde{z}_{\min}(t_1)=\widetilde{z}(\widetilde{x}_1,t_1)$. 
Then $f_t$, $\widetilde{f}_t$ are both graphs on the neighborhoods of $x_1$ and $\widetilde{x}_1$. Thus by Lemma \ref{lem:graph curvature} we can write the above equation as
\begin{align*}
    \frac{d \rho}{d t}= -\frac{\dot{r}}{r \left(\dot{r}^2 +1 \right)}(\widetilde{z}_{\min}(t_1)) + \frac{\ddot{\widetilde{z}}(\widetilde{x}_1,t_1)}{r^2(\widetilde{z}_{\min}(t_1))} + \frac{\dot{r}}{r \left(\dot{r}^2 +1 \right)}(z_{\max}(t_1)) - \frac{\ddot{z}(x_1,t_1)}{r^2(z_{\max}(t_1))},
\end{align*}
and furthermore from $\ddot{\widetilde{z}}(\widetilde{x}_1,t_1) \geq 0$, $\ddot{z}(x_1,t_1) \leq 0$, we have
\begin{align*}
    \frac{d \rho}{d t} \geq -\frac{\dot{r}}{r \left(\dot{r}^2 +1 \right)}(\widetilde{z}_{\min}(t_1)) + \frac{\dot{r}}{r \left(\dot{r}^2 +1 \right)}(z_{\max}(t_1)).
\end{align*}
Now we define $G(z)=\frac{\dot{r}}{r \left(\dot{r}^2 +1 \right)}(z)$ and from Setting RS, we have
\begin{align}
\label{eq:G increasing}
    \dot{G}=\frac{1}{r^2 (\dot{r}^2 + 1)^2} \left\{ r \ddot{r} \left(1-\dot{r}^2 \right)- \dot{r}^2 \left( \dot{r}^2 +1 \right) \right\} > 0,
\end{align}
and thus $G$ is strictly increasing. Then we assume $\rho(t_1)<0$, we obtain
\begin{align*}
    \frac{d \rho}{d t}(t_1) > 0.
\end{align*}
Meanwhile at $t_2 \in (0,t_1]$ that meets 
$\min_{t \in [0,t_1]} \rho(t)=\rho(t_2) < 0$, we have
\begin{align*}
    \frac{d \rho}{d t}(t_2) \leq 0.
\end{align*}
However this is a contradiction and therefore $\rho(t) \geq 0$ for any $t \in [0,T)$.
\end{proof}

\section{Long-time existence of the flow}
\label{sec:long-time existence}

In this section we assume Setting RS and the initial condition $(\ref{initial})$ for the ambient space $\overline{M}$ and the flow $\{f_t\}_{t \in [0,T)}$.
First we compute the estimate as in \cite{ecker1991interior,cabezas2009volume,cabezas2012volume}. 
Then we will show the long-time existence.

\begin{lemma}
\label{lem:curvature ineq}

Let
\begin{align*}
    \varphi(v) = \frac{v^2}{1-k v^2}, \quad k=\frac{1}{2 \sup v^2},
\end{align*}
and define $\mathfrak{g}=\varphi(v) \, \kappa^2$. Then we have
\begin{align}
\label{eq:curvature ineq}
    \left(\partial_t - \Delta\right) \mathfrak{g} \leq
    -2k \mathfrak{g}^2 +  4 \mathcal{K} \frac{\sqrt{\varphi}}{v^3} \mathfrak{g}^{\frac{3}{2}} + 2 \overline{K} \mathfrak{g}  -\frac{1}{\varphi} \left( \partial_s \mathfrak{g} \cdot \partial_s \varphi  \right).
\end{align}
\end{lemma}

\begin{proof}
First,
\begin{align*}
    \left(\partial_t - \Delta\right)\mathfrak{g} 
    &= \kappa^2 \left(\partial_t - \Delta\right)\varphi + \varphi \left(\partial_t - \Delta\right)\kappa^2 - 2 \left( \partial_s \varphi \cdot \partial_s \kappa^2 \right) \\
    &=\kappa^2 \varphi^{\prime} \left(\partial_t - \Delta\right)v - \kappa^2 \varphi^{\prime\prime}\left( \partial_s v\right)^2 + \varphi \left(\partial_t - \Delta\right)\kappa^2 - 2 \left( \partial_s \varphi \cdot \partial_s \kappa^2 \right)
\end{align*}
holds and we have 
\begin{align*}
    -\left( \partial_s \varphi \cdot \partial_s \kappa^2 \right) 
    =-\frac{1}{\varphi} \left( \partial_s \mathfrak{g} \cdot \partial_s \varphi  \right) + \frac{1}{\varphi} (\partial_s \varphi)^2 \kappa^2.
\end{align*}
Using Young's inequality $(a b \leq \frac{a^2}{2 \varepsilon} + \frac{\varepsilon b^2}{2})$ with $\varepsilon = \varphi$, we get
\begin{align*}
    -\left( \partial_s \varphi \cdot \partial_s \kappa^2 \right)=\left(-\partial_s \varphi \cdot \kappa\right) \cdot 2\partial_s \kappa \leq \frac{1}{2\varphi}\left( \partial_s \varphi \right)^2 \kappa^2 + 2 \varphi \left( \partial_s \kappa \right)^2.
\end{align*}
From the above, we obtain the following inequality
\begin{align*}
    - 2 \left( \partial_s \varphi \cdot \partial_s \kappa^2 \right) 
    &\leq
    -\frac{1}{\varphi} \left( \partial_s \mathfrak{g} \cdot \partial_s \varphi  \right) + 2\varphi \left( \partial_s \kappa \right)^2 + \frac{3}{2\varphi} (\partial_s \varphi)^2 \kappa^2,
\end{align*}
and we have
\begin{align*}
    \left(\partial_t - \Delta\right)\mathfrak{g} 
    &\leq 
    \kappa^2 \varphi^{\prime} \left(\partial_t - \Delta\right)v - \kappa^2 \varphi^{\prime\prime}\left( \partial_s v\right)^2 + \varphi \left(\partial_t - \Delta\right)\kappa^2 \\
    &\quad -\frac{1}{\varphi} \left( \partial_s \mathfrak{g} \cdot \partial_s \varphi  \right) + 2\varphi \left( \partial_s \kappa \right)^2 + \frac{3}{2\varphi} (\partial_s \varphi)^2 \kappa^2.
\end{align*}
Using $\partial_s v = \frac{\partial_s \varphi}{\varphi^{\prime}}$ and the evolution equations for $v$ and $\kappa^2$, we can write
\begin{align}
\label{eq:prepare ineq}
    \left(\partial_t - \Delta\right)\mathfrak{g} 
    &\leq \mathfrak{S} - \kappa^2 \left( \frac{2}{\varphi^{\prime}v} + \frac{\varphi^{\prime\prime}}{{\varphi^{\prime}}^2} - \frac{3}{2\varphi} \right) \left( \partial_s \varphi \right)^2 -\frac{1}{\varphi} \left( \partial_s \mathfrak{g} \cdot \partial_s \varphi  \right),
\end{align}
where $\mathfrak{S}$ is defined by
\begin{align*}
    \mathfrak{S}= \left( 2 \mathcal{K}^2 + \overline{K} \right) \left( v - \frac{1}{v} \right) \kappa^2 \varphi^{\prime}  - \frac{1}{v} \left(\mathcal{K}-\kappa v\right)^2 \kappa^2 \varphi^{\prime} + 2 \varphi \kappa^2\left(\kappa^2 + \overline{K} \right).
\end{align*}
Now from $\varphi^{\prime}=2 \frac{\varphi^2}{v^3}$, we have
\begin{equation}
\label{eq:S}
    \begin{split}
    \mathfrak{S} 
    &=
    2\frac{\kappa^2 \varphi^2}{v^3} \left( 2 \mathcal{K}^2 + \overline{K} \right) \left( v - \frac{1}{v} \right) -2\frac{\kappa^2 \varphi^2}{v^4} \left( \mathcal{K}^2 -2 \mathcal{K} \kappa v + \kappa^2 v^2 \right) + 2 \kappa^4 \varphi + 2 \mathfrak{g} \overline{K}
    \\ 
    &\leq
    4 \mathcal{K} \frac{\kappa^3 \varphi^2}{v^3} -2\frac{\kappa^4 \varphi^2}{v^2} + 2 \kappa^4 \varphi + 2 \mathfrak{g} \overline{K}
    \\
    &=
    \left( - \frac{2}{v^2} + \frac{2}{\varphi} \right) \mathfrak{g}^2 + 4 \mathcal{K} \frac{\sqrt{\varphi}}{v^3} \mathfrak{g}^{\frac{3}{2}} + 2 \overline{K} \mathfrak{g}
    \\
    &= -2 k \mathfrak{g}^2 + 4 \mathcal{K} \frac{\sqrt{\varphi}}{v^3} \mathfrak{g}^{\frac{3}{2}} + 2 \overline{K} \mathfrak{g}.
    \end{split}
\end{equation}
(\ref{eq:prepare ineq}), (\ref{eq:S}) and
\begin{align*}
     \left( \frac{2}{\varphi^{\prime}v} + \frac{\varphi^{\prime\prime}}{{\varphi^{\prime}}^2} - \frac{3}{2\varphi} \right) = \frac{\varphi - v^2}{2 \varphi^2} > 0
\end{align*}
prove the inequality (\ref{eq:curvature ineq}).

\end{proof}

Before we prove the main theorem, we give three lemmas.

\begin{lemma}
\label{lem:z bounded}
If $T<\infty$, then $z$ is bounded below.
\end{lemma}
\begin{proof}
We set $z_{\min}(t)=\min_{x \in \mathbb{S}^1} z(x,t)$. By the evolution equation for $z$, we obtain
\begin{align*}
    \frac{d z_{\min}}{d t} \geq - G(z_{\min}(t)),
\end{align*}
where $G$ is the function defined in the proof of Theorem $\ref{thm:comparison}$.
Since $G$ is strictly increasing (from $(\ref{eq:G increasing})$) and $z_{max}$ is strictly decreasing, we get
\begin{align*}
    -z_{\min}(t) 
    &\leq - z_{\min}(0) + \int_{0}^{t} G(z_{\min}(t)) d t \\
    &\leq - z_{\min}(0) + \int_{0}^{t} G(z_{\max}(t)) d t \\
    &\leq - z_{\min}(0) + \int_{0}^{t} G(z_{\max}(0)) d t \\
    &\leq - z_{\min}(0) + G(z_{\max}(0)) T, \\
\end{align*}
and thus $z$ is bounded below.
\end{proof}

From this lemma, the flow $\{ f_t \}_{t \in [0,T)}$ is included in a bounded region of $\overline{M}$ if $T < \infty$.

\begin{lemma}
\label{lem:time bounded}
If $T<\infty$, then $\max_{x \in \mathbb{S}^1}\kappa^2(x,t) \to \infty$  $(t \to T)$.
\end{lemma}
\begin{proof}
This lemma follows Lemma \ref{lem:z bounded} and the arguments of Theorem 7.1 of \cite{huisken1986contracting}.
\end{proof}

\begin{lemma}
\label{lem:K bounded}
$\mathcal{K}(z) \to 0$  $(z \to - \infty)$.
\end{lemma}
\begin{proof}
We prove $\eta(z)=\frac{\dot{r}(z)}{r(z)} \to 0$ $(z \to -\infty)$. 
From Setting RS, we have $r \ddot{r} - \dot{r}^2 > \dot{r}^2$ and compute the derivative of $\eta$, we get 
\begin{align*}
    \dot{\eta} = \frac{r \ddot{r} - \dot{r}^2}{r^2} > \left( \frac{\dot{r}}{r} \right)^2= \eta^2.
\end{align*}
By using this inequality, we have
\begin{align*}
    -\frac{\dot{\eta}}{\eta^2}=\frac{d}{d z} \left( \frac{1}{\eta} \right) < -1
\end{align*}
and integrate it, then
\begin{align*}
    \frac{1}{\eta(z_0)}-\frac{1}{\eta(z)} \leq z - z_0 \to -\infty \quad (z \to -\infty).
\end{align*}
Thus we have $\eta(z) \to 0$ and get $\mathcal{K}(z) = \frac{\eta(z)}{\sqrt{\dot{r}(z)^2 + 1}} \to 0$ $(z \to -\infty)$.
\end{proof}

Finally we show the main theorem.

\begin{proof}[\rm{\bf{Proof of Main Theorem}}]

$(1)$ 
We will prove by contradiction and we suppose that $T<\infty$. 
For any $t_0 < T$, we choose $(x_1,t_1) \in \mathbb{S}^1 \times [0,t_0]$ that satisfies $\mu(t_0) = \max_{(x,t) \in \mathbb{S}^1 \times [0,t_0]}\mathfrak{g}(x,t)=\mathfrak{g}(x_1,t_1)$. 
We can assume $t_1 > 0$. Then from $(\ref{eq:curvature ineq})$ and $\overline{K}<0$, for some positive constant $D$ independent of $t_0$ we get
\begin{align*}
    \mathfrak{g}(x_1,t_1)^2 \leq
    D \,\mathfrak{g}(x_1,t_1)^{\frac{3}{2}}.
\end{align*}
Namely we have
\begin{align}
\label{eq:contradiction}
    \mu(t_0)^2 \leq
    D \, \mu(t_0)^{\frac{3}{2}}.
\end{align}
Now if $t_0 \to T$, then $\mu(t_0)$ would blow up by Lemma\ref{lem:time bounded}, that is, it leads to a contradiction with $(\ref{eq:contradiction})$. 
Thus we get $T=\infty$.

$(2)$ 
We consider the flow $\{\widetilde{f}_t\}_{t \in [0,T)}$ (defined by $(\ref{eq:standard flow})$) with the initial condition $\widetilde{z}_0(x) = z_{\max}(0)$. 
Then for this flow we have $\widetilde{z}_t \to -\infty$ $( t \to \infty )$, and 
therefore $z_t \to - \infty$ follows from the comparison principle.

$(3)$ 
We set $\mathfrak{g}_{\max}(t) = \max_{x \in \mathbb{S}^1} \mathfrak{g}(x,t)$. From the inequality $(\ref{eq:curvature ineq})$, we have 
\begin{align*}
    \frac{d \mathfrak{g}_{\max}}{d t} \leq
    \left( -2k \mathfrak{g}_{\max} +  4 \mathcal{K} \frac{\sqrt{\varphi}}{v^3} \sqrt{\mathfrak{g}_{\max}} + 2 \overline{K} \right) \mathfrak{g}_{\max}.
\end{align*}
$\mathfrak{g}_{\max}$ is bounded from $(\ref{eq:contradiction})$, so we suppose that $\mathfrak{g}_{\max} \to \varepsilon > 0$ $(t \to \infty)$, 
\begin{align*}
    -2k \mathfrak{g}_{\max} +  4 \mathcal{K} \frac{\sqrt{\varphi}}{v^3} \sqrt{\mathfrak{g}_{\max}} + 2 \overline{K} \leq -\widetilde{D} < 0
\end{align*}
holds with some positive constant $\widetilde{D}$ by Lemma \ref{lem:K bounded} and $\overline{K} < 0$. Thus 
\begin{align*}
    \frac{d \mathfrak{g}_{\max}}{d t} \leq - \widetilde{D} \mathfrak{g}_{\max},
\end{align*}
and then we get
\begin{align*}
    0 < \mathfrak{g}_{\max}(t) \leq \mathfrak{g}_{\max}(t_0) e^{-\widetilde{D}(t-t_0)} \to 0 \quad (t \to \infty).
\end{align*}
However this is a contradiction, and therefore $\mathfrak{g}_{\max} \to 0$, that is, $\kappa_t \to 0$.

\end{proof}

\section{Acknowledgements}

I would like to thank my supervisor Naoyuki Koike for helpful conversations and support.


\vspace{0.5truecm}

{\small 
\rightline{Department of Mathematics, Graduate School of Science}
\rightline{Tokyo University of Science, 1-3 Kagurazaka}
\rightline{Shinjuku-ku, Tokyo 162-8601, Japan}
\rightline{(naotoshifujihara@gmail.com)}
}

\end{document}